\documentclass[10pt, leqno]{amsart}
\usepackage{amsfonts,amssymb}

\usepackage{amsmath}
\usepackage{amsthm}
\usepackage{amsrefs}
\usepackage{qsymbols}
\usepackage{latexsym}
\usepackage{chngcntr}
\usepackage[noadjust]{cite}
\usepackage{paralist}
\usepackage{esint}

\newtheorem{theorem}{Theorem}[section]

\newtheorem{lemma}[theorem]{Lemma}

\newtheorem{corollary}[theorem]{Corollary}
\theoremstyle{definition}
\newtheorem{definition}[theorem]{Definition}
\newtheorem{remark}[theorem]{Remark}

\newtheorem{assumption}[theorem]{Assumption}

\newcommand{\IR}{\mathbb{R}}
\newcommand{\IC}{\mathbb{C}}
\newcommand{\IN}{\mathbb{N}}

\newcommand{\IP}{\mathbb{P}}

\newcommand{\cH}{\mathcal{H}}

\newcommand{\cC}{\mathcal{C}}
\newcommand{\cL}{\mathcal{L}}

\renewcommand{\L}{\mathrm{L}}
\newcommand{\C}{\mathrm{C}}

\renewcommand{\H}{\mathrm{H}}

\renewcommand{\S}{\mathrm{S}}

\newcommand{\fa}{\mathfrak{a}}
\newcommand{\fb}{\mathfrak{b}}

\newcommand{\e}{\mathrm{e}}
\newcommand{\ii}{\mathrm{i}}
\renewcommand{\d}{\mathrm{d}}

\newcommand{\loc}{\mathrm{loc}}
\renewcommand\Re{\operatorname{Re}}

\newcommand{\Lop}{\mathcal{L}}

\newcommand{\divergence}{\operatorname{div}}

\DeclareMathOperator{\supp}{supp}
\DeclareMathOperator{\dist}{dist}

\DeclareMathOperator{\Id}{Id}

\DeclareMathOperator{\dom}{\mathcal{D}}

\numberwithin{equation}{section}

\title[Off-diagonal decay for the generalized Stokes semigroup]{On off-diagonal decay properties of the generalized Stokes semigroup with bounded measurable coefficients}
\author{Patrick Tolksdorf}
\address{Institut f\"ur Mathematik, Johannes Gutenberg-Universit\"at Mainz, Staudingerweg 9, 55099 Mainz, Germany}
\email{tolksdorf@uni-mainz.de}

\subjclass[2010]{}
\date{\today}
\thanks{}

\begin{document}
\begin{abstract}
We investigate off-diagonal decay properties of the generalized Stokes semigroup with bounded measurable coefficients on $\L^2_{\sigma} (\IR^d)$. Such estimates are well-known for elliptic equations in the form of pointwise heat kernel bounds and for elliptic systems in the form of integrated off-diagonal estimates. On our way to unveil this off-diagonal behavior we prove resolvent estimates in Morrey spaces $\L^{2 , \nu} (\IR^d)$ with $0 \leq \nu < 2$.
\end{abstract}
\maketitle

\section{Introduction}

\noindent In this note we study decay properties of the resolvent as well as the associated semigroup of the generalized Stokes operator $A$ on $\L^2_{\sigma} (\IR^d)$. This operator is formally given by
\begin{align*}
 A u = - \divergence(\mu \nabla u) + \nabla \phi, \quad \divergence(u) = 0 \quad \text{in} \quad \IR^d.
\end{align*}
Here, the function $u$ denotes a fluid velocity and $\phi$ denotes the to the generalized Stokes equations associated pressure function. The matrix of coefficients is merely supposed to be essentially bounded and ellipticity is enforced by a G\r{a}rding type inequality. \par
If the elliptic counterpart $L u = - \divergence (\mu \nabla u)$ is considered, then certain off-diagonal decay properties of the corresponding heat semigroup are well-known. For example, if $L$ represents an elliptic \textit{equation} with \textit{real coefficients}, then the kernel $k_t (\cdot , \cdot)$ of the associated heat semigroup $(\e^{- t L})_{t \geq 0}$ satisfies heat kernel bounds
\begin{align*}
 \lvert k_t (x , y) \rvert \leq C t^{- \frac{d}{2}} \e^{- c \frac{\lvert x - y \rvert^2}{t}}.
\end{align*}
It is well-known that if $L$ represents an elliptic \textit{system} with \textit{real/complex coefficients} these heat kernel bounds seize to be valid~\cite{Mazya_Nazarov_Plamenevskii, Davies, Frehse}. The natural substitute for heat kernel bounds for elliptic systems are so-called off-diagonal estimates. The simplest version are $\L^2$ off-diagonal estimates for the heat semigroup, its gradient, or also for $L$ applied to the heat semigroup and are of the form
\begin{align}
\label{Eq: Elliptic off-diagonal}
 \| \e^{- t L} f \|_{\L^2 (F)} + t^{\frac{1}{2}} \| \nabla \e^{- t L} f \|_{\L^2 (F)} + t \| L \e^{- t L} f \|_{\L^2 (F)} \leq C \e^{- c \frac{\dist(E , F)^2}{t}} \| f \|_{\L^2 (E)},
\end{align}
where $E , F \subset \IR^d$ are closed subsets and $f \in \L^2 (\IR^d)$ has its support in $E$. Such estimates build the foundation for many deep results in the harmonic analysis of elliptic operators with rough coefficients as can be seen, e.g., in the seminal works on the Kato square root problem~\cite{Auscher_Hofmann_Lacey_McIntosh_Tchamitchian} as well as on mapping properties of Riesz transforms on $\L^p$-spaces~\cite{Auscher_Memoirs} or the well-posedness results of Navier--Stokes like equations with initial data in $\mathrm{BMO}^{-1}$~\cite{Auscher_Frey} in the spirit of Koch and Tataru~\cite{Koch_Tataru}. \par
The spirit of how these off-diagonal estimates~\eqref{Eq: Elliptic off-diagonal} are used is as follows. For example, one might be interested in estimating an expression that involves $\e^{- t L} f$ in some sense. One then decomposes $\IR^d$ into carefully chosen disjoint sets, e.g., into annuli of the form $\cC_k := \overline{B(x_0 , 2^{k + 1} r)} \setminus B(x_0 , 2^k r)$, $k \in \IN$, and $\cC_0 := \overline{B(x_0 , 2 r)}$. Then one would estimate by virtue of~\eqref{Eq: Elliptic off-diagonal}
\begin{align}
\label{Eq: Off-diagonal calculation}
\begin{aligned}
 \| \e^{- t L} f \|_{\L^2 (B(x_0 , r))} &\leq \sum_{k = 0}^{\infty} \| \e^{- t L} \chi_{\cC_k} f \|_{\L^2 (B(x_0 , r))} \\
 &\leq C \| f \|_{\L^2 (B(x_0 , 2 r))} + C \sum_{k = 0}^{\infty} \e^{- c \frac{r^2}{t} 2^{2 k}} \| f \|_{\L^2 (B(x_0 , 2^{k + 1} r))}
\end{aligned}
\end{align}
and proceed with the proof in a certain manner, depending on the particular situation. \par
The question, whose study we want to initiate here, is whether or not the generalized Stokes semigroup $(\e^{- t A})_{t \geq 0}$ satisfies off-diagonal decay estimates and if so, how they look like. The main problem is already, that in a calculation of the form~\eqref{Eq: Off-diagonal calculation} one multiplies $f$ by a characteristic function. This in general destroys the solenoidality of the function $f$. Thus, if one wants to perform such an operation, one is urged to think about \textit{how to extend} $\e^{- t A}$ to all of $\L^2 (\IR^d)$. In many situations, the gold standard is to extend $\e^{- t A}$ to all of $\L^2 (\IR^d)$ by studying $\e^{- t A} \IP$, where $\IP$ denotes the Helmholtz projection on $\L^2 (\IR^d)$. Thus, in order to imitate the calculation performed in~\eqref{Eq: Off-diagonal calculation} one would need that off-diagonal bounds for $\e^{- t A} \IP$ are valid. However, estimates of the form
\begin{align}
\label{Eq: Wrong off-diagonal}
 \| \e^{- t A} \IP f \|_{\L^2 (F)} \leq g(\tfrac{\dist(E , F)^2}{t}) \| f \|_{\L^2 (E)}
\end{align}
with $g : [0 , \infty) \to [0 , \infty)$ satisfying $\lim_{x \to \infty} g(x) = 0$ and $f$ being supported in $E$ are in general \textit{wrong}. The reason is simple: fix any closed subset $E \subset \IR^d$ and let $F \subset \IR^d$ denote any other closed set that satisfies $\dist(E , F) > 0$. On the one hand, since $(\e^{- t A})_{t \geq 0}$ is strongly continuous on $\L^2_{\sigma} (\IR^d)$ with $\e^{- 0 A} f = f$ one has that
\begin{align*}
 \lim_{t \to 0} \| \e^{- t A} \IP f \|_{\L^2 (F)} = \| \IP f \|_{\L^2 (F)}.
\end{align*}
On the other hand~\eqref{Eq: Wrong off-diagonal} together with the condition on $g$ implies that $\| \IP f \|_{\L^2 (F)} = 0$. This implies that $\supp(\IP f) \subset E$ whenever $f \in \L^2 (\IR^d)$ with $\supp (f) \subset E$. As a consequence, the Helmholtz projection would be a local operator, which is known to be wrong. \par
Thus, in order to establish off-diagonal bounds for the generalized Stokes semigroup, one either needs to find the correct extension of the generalized Stokes semigroup to all of $\L^2 (\IR^d)$ or one needs to avoid arguments that destroy the solenoidality of $f$. In particular, this rules out standard proofs of off-diagonal estimates that are used in the elliptic situation as, e.g., Davies' trick~\cite{Davies_trick}. \par
The main result of this note is an estimate of the type~\eqref{Eq: Off-diagonal calculation}. Let us introduce some notation to state this in a precise form: 

\begin{assumption}
\label{Ass: Coefficients}
The coefficients $\mu = (\mu_{\alpha \beta}^{i j})_{\alpha , \beta , i , j = 1}^d$ with $\mu_{\alpha \beta}^{i j} \in \L^{\infty} (\IR^d ; \IC)$ for all $1 \leq \alpha , \beta , i , j \leq d$ satisfy for some $\mu_{\bullet} , \mu^{\bullet} > 0$ the inequalities
\begin{align}
\label{Eq: Ellipticity}
 \Re \sum_{\alpha , \beta , i , j = 1}^d \int_{\IR^d} \mu^{i j}_{\alpha \beta} \partial_{\beta} u_j \overline{\partial_{\alpha} u_i} \, \d x \geq \mu_{\bullet} \| \nabla u \|_{\L^2}^2 \qquad (u \in \H^1 (\IR^d ; \IC^d))
\end{align}
and
\begin{align}
\label{Eq: Boundedness}
 \max_{1 \leq i , j , \alpha , \beta \leq d} \| \mu^{i j}_{\alpha \beta} \|_{\L^{\infty}} \leq \mu^{\bullet}.
\end{align}
\end{assumption} 

The operator $A$ is realized on $\L^2_{\sigma} (\IR^d) := \{ f \in \L^2 (\IR^d ; \IC^d) : \divergence(f) = 0 \}$ as follows. Let $\H^1_{\sigma} (\IR^d) := \{ f \in \H^1 (\IR^d ; \IC^d) : \divergence(f) = 0 \}$. Define the sesquilinear form
\begin{align*}
 \fa : \H^1_{\sigma} (\IR^d) \times \H^1_{\sigma} (\IR^d) \to \IC, \quad (u , v) \mapsto \sum_{\alpha , \beta , i , j = 1}^d \int_{\IR^d} \mu_{\alpha \beta}^{i j} \partial_{\beta} u_j \overline{\partial_{\alpha} v_i} \, \d x
\end{align*}
and define the domain of $A$ on $\L^2_{\sigma} (\IR^d)$ as
\begin{align*}
 \dom(A) := \bigg\{ u \in \H^1_{\sigma} (\IR^d) : \, \exists f \in \L^2_{\sigma} (\IR^d) \text{ such that } \forall v \in \H^1_{\sigma} (\IR^d) \text{ it holds } \fa (u , v) = \int_{\IR^d} f \cdot \overline{v} \, \d x  \bigg\}\cdotp
\end{align*}

The main result of this note is the following theorem:

\begin{theorem}
\label{Thm: Generalized Stokes semigroup}
Let $d \geq 2$ and let $\mu$ satisfy Assumption~\ref{Ass: Coefficients} with constants $\mu^{\bullet} , \mu_{\bullet} > 0$. For all $\nu \in (0 , 2)$ there exists $C > 0$ such that for all $x_0 \in \IR^d$, $r > 0$, $t > 0$, and $f \in \L^2_{\sigma} (\IR^d)$ it holds
\begin{align*}
 \| \e^{- t A} f \|_{\L^2 (B(x_0 , r))} &+ t \| A \e^{- t A} f \|_{\L^2 (B(x_0 , r))} \\
 &\qquad\qquad \leq C \| f \|_{\L^2 (B(x_0 , 2 r))} + C \sum_{k = 2}^{\infty} \bigg( 1 + \frac{2^{2 k} r^2}{t} \bigg)^{- \frac{\nu}{4}} \| f \|_{\L^2 (B(x_0 , 2^k r))}.
\end{align*}
Moreover, for all $F \in \L^2 (\IR^d ; \IC^{d \times d})$ it holds
\begin{align*}
 t^{\frac{1}{2}} \| \e^{- t A} \IP \divergence(F )\|_{\L^2 (B(x_0 , r))} \leq C \| F \|_{\L^2 (B(x_0 , 2 r))} + C \sum_{k = 2}^{\infty} \bigg( 1 + \frac{2^{2 k} r^2}{t} \bigg)^{- \frac{\nu}{4}} \| F \|_{\L^2 (B(x_0 , 2^k r))}.
\end{align*}
In both estimates, the constant $C$ only depends on $\mu_{\bullet}$, $\mu^{\bullet}$, $d$, and $\nu$.
\end{theorem}

As a corollary of Theorem~\ref{Thm: Generalized Stokes semigroup} one derives the following off-diagonal estimates.

\begin{corollary}
\label{Cor: Generalized Stokes semigroup}
Let $d \geq 2$ and let $\mu$ satisfy Assumption~\ref{Ass: Coefficients} with constants $\mu^{\bullet} , \mu_{\bullet} > 0$. For all $\nu \in (0 , 2)$ there exists $C > 0$ such that for all $x_0 \in \IR^d$, $r > 0$, $k_0 \in \IN$ with $k_0 \geq 2$, $t > 0$, and $f \in \L^2_{\sigma} (\IR^d)$ with $\supp(f) \subset \overline{B(x_0 , 2^{k_0} r)} \setminus B(x_0 , 2^{k_0 - 1} r)$ it holds
\begin{align*}
 \| \e^{- t A} f \|_{\L^2 (B(x_0 , r))} + t \| A \e^{- t A} f \|_{\L^2 (B(x_0 , r))} \leq C \bigg( 1 + \frac{2^{2 {k_0}} r^2}{t} \bigg)^{- \frac{\nu}{4}} \| f \|_{\L^2 (B(x_0 , 2^{k_0} r) \setminus B(x_0 , 2^{k_0 - 1} r))}.
\end{align*}
Moreover, for all $F \in \L^2 (\IR^d ; \IC^{d \times d})$ with $\supp(F) \subset \overline{B(x_0 , 2^{k_0} r)} \setminus B(x_0 , 2^{k_0 - 1} r)$ it holds
\begin{align*}
 t^{\frac{1}{2}} \| \e^{- t A} \IP \divergence(F) \|_{\L^2 (B(x_0 , r))} \leq C \bigg( 1 + \frac{2^{2 {k_0}} r^2}{t} \bigg)^{- \frac{\nu}{4}} \| F \|_{\L^2 (B(x_0 , 2^{k_0} r) \setminus B(x_0 , 2^{k_0 - 1} r))}.
\end{align*}
In both estimates, the constant $C$ only depends on $\mu_{\bullet}$, $\mu^{\bullet}$, $d$, and $\nu$.
\end{corollary}

\section{A non-local resolvent estimate}

\noindent To establish Theorem~\ref{Thm: Generalized Stokes semigroup} we prove analogous estimates for the resolvent of $A$. More precisely, we are going to estimate the solution $u$ to the generalized Stokes resolvent problem
\begin{align}
\label{Eq: Resolvent problem}
 \left\{ \begin{aligned}
  \lambda u - \divergence(\mu \nabla u) + \nabla \phi &= f + \IP \divergence(F) && \text{in } \IR^d, \\
  \divergence(u) &= 0 && \text{in } \IR^d
 \end{aligned} \right.
\end{align}
for $\lambda$ in some complex sector $\S_{\omega} := \{ z \in \IC \setminus \{ 0 \} : \lvert \arg(z) \rvert < \omega \}$. Using Assumption~\ref{Ass: Coefficients} together with the lemma of Lax--Milgram, one finds some $\omega \in (\pi / 2 , \pi)$ depending on $\mu_{\bullet}$, $\mu^{\bullet}$, and $d$ such that~\eqref{Eq: Resolvent problem} is uniquely solvable for all $f \in \L^2_{\sigma} (\IR^d)$ and all $F \in \L^2 (\IR^d ; \IC^{d \times d})$. In the follwing, let us denote the solution operator to~\eqref{Eq: Resolvent problem} by $(\lambda + A)^{-1}$. The solution $u$ to~\eqref{Eq: Resolvent problem} then lies in the space $\H^1_{\sigma} (\IR^d)$ and for all $\theta \in (0 , \omega)$ there exists $C > 0$ such that for all $f \in \L^2_{\sigma} (\IR^d)$, $F \in \L^2 (\IR^d ; \IC^{d \times d})$, and all $\lambda \in \S_{\theta}$ it satisfies the resolvent estimates
\begin{align}
\label{Eq: Resolvent estimate}
 \| \lambda (\lambda + A)^{-1} f \|_{\L^2} + \lvert \lambda \rvert^{\frac{1}{2}} \| \nabla (\lambda + A)^{-1} f \|_{\L^2} + \| A (\lambda + A)^{-1} f \|_{\L^2} \leq C \| f \|_{\L^2}
\end{align}
and
\begin{align}
\label{Eq: Resolvent estimate 2}
 \lvert \lambda \rvert^{\frac{1}{2}} \| (\lambda + A)^{-1} \IP \divergence(F) \|_{\L^2} + \| \nabla (\lambda + A)^{-1} \IP \divergence(F) \|_{\L^2} \leq C \| F \|_{\L^2}.
\end{align}
The next lemma was proven in~\cite[Lem.~5.3]{Tolksdorf} and combines different types of Caccioppoli inequalities to account for the non-local pressure.

\begin{lemma}
\label{Lem: Preparation of reverse Holder}
Let $\mu$ satisfy Assumption~\ref{Ass: Coefficients} with constants $\mu^{\bullet} , \mu_{\bullet} > 0$. There exists $\omega \in (\pi / 2 , \pi)$ such that for all $\theta \in (0 , \omega)$, $f \in \L^2_{\sigma} (\IR^d)$, $F \in \L^2 (\IR^d ; \IC^{d \times d})$, and $\lambda \in \S_{\theta}$ the following holds: for $u \in \H^1_{\sigma} (\IR^d)$ defined by $u := (\lambda + A)^{-1} (f + \IP \divergence(F))$ and $x_0 \in \IR^d$ and $r_0 > 0$ there exists a decomposition of $u$ of the form $u = u_1 + u_2$ with $u_1 \in \H^1 (B(x_0 , r_0) ; \IC^d)$ satisfying $\divergence(u_1) = 0$ and $u_2 \equiv u$ in $\IR^d \setminus B(x_0 , r_0)$ and there exists $\phi_1 \in \L^2 (B(x_0 , r_0))$ and $C > 0$ such that for any ball $B \subset \IR^d$ of radius $r > 0$ with $2B \subset B(x_0 , r_0)$ we have
\begin{align}
\label{Eq: Second estimate}
\begin{aligned}
 &\lvert \lambda \rvert^3 r^2 \int_{B} \lvert u_2 \rvert^2 \, \d x + \lvert \lambda \rvert^2 r^2 \int_{B} \lvert \nabla u_2 \rvert^2 \; \d x \\
  &\quad \leq C \bigg\{ \sum_{\ell = 0}^{\infty} 2^{- \ell d - \ell} \int_{2^{\ell} B} \big(\lvert \lambda u \rvert^2 + \lvert f \rvert^2 + \lvert \lvert \lambda \rvert^{\frac{1}{2}} F \rvert^2\big) \, \d x + \int_{2 B} \lvert \lambda u_1 \rvert^2 \, \d x + \int_{2 B} \lvert \lvert \lambda \rvert^{\frac{1}{2}} \phi_1 \rvert^2 \, \d x \bigg\}\cdotp
\end{aligned}
\end{align}
Moreover, $u_1$ and $\phi_1$ satisfy for some $C > 0$
\begin{align}
\label{Eq: First estimate u}
\begin{aligned}
 \lvert \lambda \rvert \| u_1 \|_{\L^2 (B(x_0 , r_0))} + \lvert \lambda \rvert^{\frac{1}{2}} \| \nabla u_1 \|_{\L^2 (B (x_0 , r_0))} &+ \lvert \lambda \rvert^{\frac{1}{2}} \| \phi_1 \|_{\L^2 (B(x_0 , r_0))} \\
 &\quad\leq C \big( \| f \|_{\L^2 (B (x_0 , r_0))} + \lvert \lambda \rvert^{\frac{1}{2}} \| F \|_{\L^2 (B(x_0 , r_0))} \big).
\end{aligned}
\end{align}
In both inequalities, the constant $C$ only depends on $d$, $\theta$, $\mu_{\bullet}$, and $\mu^{\bullet}$. Moreover, $\omega$ only depends on $d$, $\mu_{\bullet}$, and $\mu^{\bullet}$.
\end{lemma}

This lemma can be used to prove the following non-local resolvent estimate.

\begin{theorem}
\label{Thm: Nonlocal resolvent estimate}
Let $\mu$ satisfy Assumption~\ref{Ass: Coefficients} with constants $\mu^{\bullet} , \mu_{\bullet} > 0$. There exists $\omega \in (\pi / 2 , \pi)$ such that for all $\theta \in (0 , \omega)$ and all $\nu \in (0 , 2)$ there exists a constant $C > 0$ such that for all $\lambda \in \S_{\theta}$, $f \in \L^2_{\sigma} (\IR^d)$, and $F \in \L^2 (\IR^d ; \IC^{d \times d})$ the unique solution $u \in \H^1_{\sigma} (\IR^d)$ to~\eqref{Eq: Resolvent problem} satisfies
\begin{align*}
 \sum_{k = 0}^{\infty} 2^{- \nu k} \int_{B(x_0 , 2^k r)} \lvert \lambda u \rvert^2 \, \d x \leq C \sum_{k = 0}^{\infty} 2^{- \nu k} \int_{B(x_0 , 2^k r)} \big( \lvert f \rvert^2 + \lvert \lvert \lambda \rvert^{\frac{1}{2}} \rvert F \rvert^2 \big) \, \d x.
\end{align*}
Here, the constant $C$ only depends on $d$, $\theta$, $\nu$, $\mu_{\bullet}$, and $\mu^{\bullet}$ and $\omega$ only depends on $d$, $\mu_{\bullet}$, and $\mu^{\bullet}$.
\end{theorem}

\begin{proof}
We use the decomposition of $u$ from Lemma~\ref{Lem: Preparation of reverse Holder} as follows. Fix $k \in \IN_0$ and let $\ell_0 \in \IN$ to be determined. Let $u_{1 , k}$, $u_{2 , k}$, and $\phi_{1 , k}$ be the functions determined by Lemma~\ref{Lem: Preparation of reverse Holder} with $r_0 := 2^{k + \ell_0 + 1} r$. Now, we proceed by applying H\"older's inequality, then increase the domain of integration, and use Sobolev's embedding to obtain for $q > 1$ with
\begin{align}
\label{Eq: Condition q}
 \frac{1}{2} - \frac{1}{2 q} \leq \frac{1}{d}
\end{align}
the inequalities
\begin{align*}
 &\int_{B(x_0 , 2^k r)} \lvert u_{2 , k} \rvert^2 \, \d x \leq \lvert B(x_0 , 2^k r) \rvert^{1 - \frac{1}{q}} \bigg( \int_{B(x_0 , 2^k r)} \lvert u_{2 , k} \rvert^{2 q} \, \d x \bigg)^{\frac{1}{q}} \\
 &\qquad\leq \frac{\lvert B(x_0 , 2^k r) \rvert^{1 - \frac{1}{q}}}{\lvert B(x_0 , 2^{ k + \ell_0} r) \rvert^{- \frac{1}{q}}} \bigg( \fint_{B(x_0 , 2^{k + \ell_0} r)} \lvert u_{2 , k} \rvert^{2 q} \, \d x \bigg)^{\frac{1}{q}} \\
 &\qquad\leq C \frac{\lvert B(x_0 , 2^k r) \rvert^{1 - \frac{1}{q}}}{\lvert B(x_0 , 2^{ k + \ell_0} r) \rvert^{1 - \frac{1}{q}}}\bigg\{ \int_{B(x_0 , 2^{k + \ell_0} r)} \lvert u_{2 , k} \rvert^2 \, \d x + (2^{k + \ell_0} r)^2 \int_{B(x_0 , 2^{k + \ell_0} r)} \lvert \nabla u_{2 , k} \rvert^2 \, \d x \bigg\} \\
 &\qquad= C 2^{- \ell_0 d (1 - \frac{1}{q})} \bigg\{ \int_{B(x_0 , 2^{k + \ell_0} r)} \lvert u_{2 , k} \rvert^2 \, \d x + (2^{k + \ell_0} r)^2 \int_{B(x_0 , 2^{k + \ell_0} r)} \lvert \nabla u_{2 , k} \rvert^2 \, \d x \bigg\}\cdotp
\end{align*}
Notice that the constant $C > 0$ in the previous estimate only depends on $d$ and $q$. Now, use this estimate together with $u_{2 , k} = u - u_{1 , k}$ and~\eqref{Eq: Second estimate} and~\eqref{Eq: First estimate u} to deduce
\begin{align*}
 &\int_{B(x_0 , 2^k r)} \lvert \lambda u \rvert^2 \, \d x \\
 &\qquad\leq 2 \int_{B(x_0 , 2^k r)} \lvert \lambda u_{1 , k} \rvert^2 \, \d x + 2 \int_{B(x_0 , 2^k r)} \lvert \lambda u_{2 , k} \rvert^2 \, \d x \\
 &\qquad\leq 2 \int_{B(x_0 , 2^k r)} \lvert \lambda u_{1 , k} \rvert^2 \, \d x \\
 &\qquad\qquad + \lvert \lambda \rvert^2 C 2^{- \ell_0 d (1 - \frac{1}{q})} \bigg\{ \int_{B(x_0 , 2^{k + \ell_0} r)} \lvert u_{2 , k} \rvert^2 \, \d x + (2^{k + \ell_0} r)^2 \int_{B(x_0 , 2^{k + \ell_0} r)} \lvert \nabla u_{2 , k} \rvert^2 \, \d x \bigg\} \\
 &\qquad\leq 2 \int_{B(x_0 , 2^k r)} \lvert \lambda u_{1 , k} \rvert^2 \, \d x \\
 &\qquad\qquad + \lvert \lambda \rvert^2 C 2^{- \ell_0 d (1 - \frac{1}{q})} \int_{B(x_0 , 2^{k + \ell_0} r)} \lvert u - u_{1 , k} \rvert^2 \, \d x \\
 &\qquad\qquad + C 2^{- \ell_0 d (1 - \frac{1}{q})} \bigg\{ \sum_{\ell = 0}^{\infty} 2^{- \ell d - \ell} \int_{B(x_0 , 2^{ k + \ell + \ell_0} r)} \big(\lvert \lambda u \rvert^2 + \lvert f \rvert^2 + \lvert \lvert \lambda \rvert^{\frac{1}{2}} F \rvert^2\big) \, \d x \\
 &\qquad\qquad + \int_{B(x_0 , 2^{k + \ell_0 + 1} r)} \lvert \lambda u_{1 , k} \rvert^2 \, \d x + \int_{B (x_0 , 2^{k + \ell_0 + 1} r)} \lvert \lvert \lambda \rvert^{\frac{1}{2}} \phi_{1 , k} \rvert^2 \, \d x \bigg\} \\
  &\qquad\leq C \int_{B(x_0 , 2^{k + \ell_0 + 1} r)} \lvert f \rvert^2 \, \d x + C 2^{- \ell_0 d (1 - \frac{1}{q})} \int_{B(x_0 , 2^{k + \ell_0} r)} \lvert \lambda u \rvert^2 \, \d x \\
 &\qquad\qquad + C 2^{- \ell_0 d (1 - \frac{1}{q})} \sum_{\ell = 0}^{\infty} 2^{- \ell d - \ell} \int_{B(x_0 , 2^{ k + \ell + \ell_0} r)} \big(\lvert \lambda u \rvert^2 + \lvert f \rvert^2 + \lvert \lvert \lambda \rvert^{\frac{1}{2}} F \rvert^2\big) \, \d x.
\end{align*}
Now, multiply this inequality by $2^{- \nu k}$ and sum with respect to $k \in \IN_0$. This then delivers
\begin{align*}
 &\sum_{k = 0}^{\infty} 2^{- \nu k} \int_{B(x_0 , 2^k r)} \lvert \lambda u \rvert^2 \, \d x \\
 &\qquad\leq C 2^{- \ell_0 (d - \frac{d}{q} - \nu)} \sum_{k = 0}^{\infty} 2^{- \nu (k + \ell_0)} \int_{B(x_0 , 2^{k + \ell_0} r)} \lvert \lambda u \rvert^2 \, \d x \\
 &\qquad\qquad + C 2^{- \ell_0 (d - \frac{d}{q} - \nu)} \sum_{\ell = 0}^{\infty} 2^{\ell (\nu - d - 1)} \sum_{k = 0}^{\infty} 2^{- \nu (k + \ell + \ell_0)} \int_{B(x_0 , 2^{k + \ell + \ell_0} r)} \lvert \lambda u \rvert^2 \, \d x \\
 &\qquad\qquad + C 2^{\nu (\ell_0 + 1)} \sum_{k = 0}^{\infty} 2^{- \nu (k + \ell_0 + 1)} \int_{B(x_0 , 2^{k + \ell_0 + 1} r)} \lvert f \rvert^2 \, \d x \\
 &\qquad\qquad + C 2^{- \ell_0 (d - \frac{d}{q} - \nu)} \sum_{\ell = 0}^{\infty} 2^{\ell (\nu - d - 1)} \sum_{k = 0}^{\infty} 2^{- \nu (k + \ell + \ell_0)} \int_{B(x_0 , 2^{k + \ell + \ell_0} r)} \big( \lvert f \rvert^2 + \lvert \lvert \lambda \rvert^{\frac{1}{2}} F \rvert^2 \big) \, \d x \allowdisplaybreaks \\
 &\qquad\leq C 2^{- \ell_0 (d - \frac{d}{q} - \nu)} \sum_{k = 0}^{\infty} 2^{- \nu k} \int_{B(x_0 , 2^k r)} \lvert \lambda u \rvert^2 \, \d x \\
 &\qquad\qquad + C \sum_{k = 0}^{\infty} 2^{- \nu k} \int_{B(x_0 , 2^k r)} \big( \lvert f \rvert^2 + \lvert \lvert \lambda \rvert^{\frac{1}{2}} F \rvert^2 \big) \, \d x.
\end{align*}
Now, in order to conclude that the exponent $d - \frac{d}{q} - \nu$ is positive, we need to require further restrictions to $q$. One immediately verifies that the positivity of this exponent as well as~\eqref{Eq: Condition q} are fulfilled, whenever $q$ satisfies
\begin{align}
\label{Eq: Choice q}
 1 - \frac{2}{d} \leq \frac{1}{q} < 1 - \frac{\nu}{d}\cdotp
\end{align}
Since $\nu < 2$, such a choice is possible. Thus, fixing $q$ subject to~\eqref{Eq: Choice q} allows to choose $\ell_0$ large enough so as to absorb the $\lambda u$-term on the right-hand side to the left-hand side. Thus, there exists $C > 0$ such that
\begin{align*}
 \sum_{k = 0}^{\infty} 2^{- \nu k} \int_{B(x_0 , 2^k r)} \lvert \lambda u \rvert^2 \, \d x \leq C \sum_{k = 0}^{\infty} 2^{- \nu k} \int_{B(x_0 , 2^k r)} \big( \lvert f \rvert^2 + \lvert \lvert \lambda \rvert^{\frac{1}{2}} F \rvert^2 \big) \, \d x. &\qedhere
\end{align*}
\end{proof}

As a corollary we get that the generalized Stokes operator satisfies resolvent estimates with respect to the Morrey space norm of $\L^{2 , \nu} (\IR^d ; \IC^d)$ for all $0 \leq \nu < 2$. The definition of this Morrey space is the following:

\begin{definition}
\noindent Let $0 \leq \nu < d$ and $m \in \IN$. Define the Morrey space $\L^{2 , \nu} (\IR^d ; \IC^m)$ as the vector space of all functions $u \in \L^2_{\loc} (\IR^d ; \IC^m)$ with finite Morrey space norm
\begin{align*}
 \| u \|_{\L^{2 , \nu}} := \sup_{\substack{x_0 \in \IR^d \\ r > 0}} \bigg( r^{- \nu} \int_{B(x_0 , r)} \lvert u \rvert^2 \, \d x \bigg)^{\frac{1}{2}}.
\end{align*}
\end{definition}

\begin{corollary}
Let $\mu$ satisfy Assumption~\ref{Ass: Coefficients} with constants $\mu^{\bullet} , \mu_{\bullet} > 0$. There exists $\omega \in (\pi / 2 , \pi)$ such that for all $\theta \in (0 , \omega)$ and all $\nu \in [0 , 2)$ there exists a constant $C > 0$ such that for all $\lambda \in \S_{\theta}$, $f \in \L^2_{\sigma} (\IR^d) \cap \L^{2 , \nu} (\IR^d ; \IC^d)$, and $F \in \L^2 (\IR^d ; \IC^{d \times d}) \cap \L^{2 , \nu} (\IR^d ; \IC^{d \times d})$ the unique solution $u \in \H^1_{\sigma} (\IR^d)$ to~\eqref{Eq: Resolvent problem} satisfies
\begin{align*}
 \| \lambda u \|_{\L^{2 , \nu}} \leq C \big( \| f \|_{\L^{2 , \nu}} + \lvert \lambda \rvert^{\frac{1}{2}} \| F \|_{\L^{2 , \nu}} \big).
\end{align*}
Here, the constant $C$ only depends on $d$, $\theta$, $\nu$, $\mu_{\bullet}$, and $\mu^{\bullet}$ and $\omega$ only depends on $d$, $\mu_{\bullet}$, and $\mu^{\bullet}$.
\end{corollary}

\begin{proof}
Fix $x_0 \in \IR^d$ and $r > 0$. The estimate in Theorem~\ref{Thm: Nonlocal resolvent estimate} readily gives for some $\nu < \nu^{\prime} < 2$
\begin{align*}
 \int_{B(x_0 , r)} \lvert \lambda u \rvert^2 \, \d x \leq C \sum_{k = 0}^{\infty} 2^{- \nu^{\prime} k} \int_{B(x_0 , 2^k r)} \big( \lvert f \rvert^2 + \lvert \lvert \lambda \rvert^{\frac{1}{2}} F \rvert^2 \big) \, \d x \leq C r^{\nu} \big( \| f \|_{\L^{2 , \nu}}^2 + \lvert \lambda \rvert \| F \|_{\L^{2 , \nu}}^2 \big).
\end{align*}
Division by $r^{\nu}$ then delivers the desired estimate.
\end{proof}

\section{$\L^2$ off-diagonal decay for the resolvent}

\noindent This section is dedicated to prove a counterpart of Theorem~\ref{Thm: Generalized Stokes semigroup} for the resolvent of $A$. For this purpose, we introduce another sesquilinear form, which is connected to the Stokes problem in a ball but with Neumann boundary conditions. \par
Let $B \subset \IR^d$ denote a ball and let
\begin{align*}
 \cL^2_{\sigma} (B) := \{ f \in \L^2 (B ; \IC^d) : \divergence(f) = 0 \text{ in the sense of distributions} \}
\end{align*}
and
\begin{align*}
 \cH^1_{\sigma} (B) := \{ f \in \H^1 (B ; \IC^d) : \divergence(f) = 0 \}.
\end{align*}
Now, define the sesquilinear form
\begin{align*}
 \fb_B : \cH^1_{\sigma} (B) \times \cH^1_{\sigma} (B) \to \IC, \quad (u , v) \mapsto \sum_{\alpha , \beta , i , j = 1}^d \int_{\IR^d} \mu^{i j}_{\alpha \beta} \partial_{\beta} u_j \overline{\partial_{\alpha} v_i} \; \d x.
\end{align*}
We abuse the notation and denote the same sesquilinear form but with domain $\H^1 (B ; \IC^d) \times \H^1 (B ; \IC^d)$ again by $\fb_B$. \par
An application of Assumption~\ref{Ass: Coefficients} and the lemma of Lax--Milgram implies the existence of $\omega \in (\pi / 2 , \pi)$ such that for all $\lambda \in \S_{\omega}$, $f \in \cL^2_{\sigma} (B)$, and $F \in \L^2 (B ; \IC^{d \times d})$ the equation
\begin{align}
\label{Eq: Solenoidal distributional Neumann}
 \lambda \int_B u \cdot \overline{v} \, \d x + \fb_B (u , v) = \int_B f \cdot \overline{v} \, \d x - \sum_{\alpha , \beta = 1}^d \int_B F_{\alpha \beta} \overline{\partial_{\alpha} v_{\beta}} \, \d x \quad (v \in \cH^1_{\sigma} (B))
\end{align}
is uniquely solvable for some $u \in \cH^1_{\sigma} (B)$. Moreover, by~\cite[Rem.~5.2]{Tolksdorf}, there exists a pressure function $\phi \in \L^2 (B)$ such that
\begin{align}
\label{Eq: Variational Neumann}
 \lambda \int_B u \cdot \overline{v} \, \d x + \fb_B (u , v) - \int_B \phi \, \overline{\divergence(v)} \, \d x = \int_B f \cdot \overline{v} \, \d x - \int_B F_{\alpha \beta} \overline{\partial_{\alpha} v_{\beta}} \, \d x \quad (v \in \H^1 (B ; \IC^d))
\end{align}
holds. Furthermore, for all $\theta \in (0 , \omega)$ there exists $C > 0$ depending only on $d$, $\theta$, $\mu_{\bullet}$, and $\mu^{\bullet}$ such that for all $\lambda \in \S_{\omega}$, $f \in \cL^2_{\sigma} (B)$, and $F \in \L^2 (B ; \IC^{d \times d})$ it holds
\begin{align}
\label{Eq: Resolvent estimate Neumann}
 \| \lambda u \|_{\L^2 (B)} + \lvert \lambda \rvert^{\frac{1}{2}} \| \nabla u \|_{\L^2 (B)} + \lvert \lambda \rvert^{\frac{1}{2}} \| \phi \|_{\L^2 (B)} \leq C \big( \| f \|_{\L^2 (B)} + \lvert \lambda \rvert^{\frac{1}{2}} \| F \|_{\L^2 (B)} \big).
\end{align}

To proceed, we cite some results from~\cite{Tolksdorf}. The first result is a non-local Caccioppoli inequality for the generalized Stokes resolvent and can be found in~\cite[Thm.~1.2]{Tolksdorf}.

\begin{theorem}
\label{Thm: Non-local Caccioppoli}
Let $\mu$ satisfy Assumption~\ref{Ass: Coefficients} for some constants $\mu_{\bullet} , \mu^{\bullet} > 0$. Then there exists $\omega \in (\pi / 2 , \pi)$ such that for all $\theta \in (0 , \omega)$ and all $0 < \nu < d + 2$ there exists $C > 0$ such that for all $\lambda \in \S_{\theta}$, $f \in \L^2_{\sigma} (\IR^d)$, $F \in \L^2 (\IR^d ; \IC^{d \times d})$ the solution $u \in \H^1_{\sigma} (\IR^d)$ to
\begin{align*}
 \lambda \int_{\IR^d} u \cdot \overline{v} \, \d x + \fa (u , v) = \int_{\IR^d} f \cdot \overline{v} \, \d x - \sum_{\alpha , \beta = 1}^d \int_{\IR^d} F_{\alpha \beta} \, \overline{\partial^{\alpha} v_{\beta}} \, \d x \qquad (v \in \H^1_{\sigma} (\IR^d))
\end{align*}
satisfies for all balls $B = B(x_0 , r)$ and all sequences $(c_k)_{k \in \IN_0}$ with $c_k \in \IC^d$
\begin{align*}
 &\lvert \lambda \rvert \sum_{k = 0}^{\infty} 2^{- \nu k} \int_{B(x_0 , 2^k r)} \lvert u \rvert^2 \, \d x + \sum_{k = 0}^{\infty} 2^{- \nu k} \int_{B(x_0 , 2^k r)} \lvert \nabla u \rvert^2 \, \d x \\
 &\qquad\qquad \leq \frac{C}{r^2} \sum_{k = 0}^{\infty} 2^{- (\nu + 2) k} \int_{B(x_0 , 2^{k + 1} r)} \lvert u + c_k \rvert^2 \, \d x + \lvert \lambda \rvert \sum_{k = 0}^{\infty} \lvert c_k \rvert 2^{- \nu k} \int_{B(x_0 , 2^{k + 1} r)} \lvert u \rvert \, \d x \\
 &\qquad\qquad\qquad + \frac{C}{\lvert \lambda \rvert} \sum_{k = 0}^{\infty} 2^{- \nu k} \int_{B(x_0 , 2^{k + 1} r)} \lvert f \rvert^2 \, \d x + C \sum_{k = 0}^{\infty} 2^{- \nu k} \int_{B(x_0 , 2^{k + 1} r)} \lvert F \rvert^2 \, \d x.
\end{align*}
The constant $\omega$ only depends on $\mu_{\bullet}$, $\mu^{\bullet}$, and $d$ and $C$ depends on $\mu_{\bullet}$, $\mu^{\bullet}$, $d$, $\theta$, and $\nu$.
\end{theorem}

The second result is an estimate on the pressure function $\phi$ that appears in~\eqref{Eq: Resolvent problem} and can be found in~\cite[Lem.~2.1]{Tolksdorf}. To formulate this lemma, we adopt the notation $\cC_k := \overline{B(x_0 , 2^k r)} \setminus B(x_0 , 2^{k - 1} r)$ for $k \in \IN$ and write $\phi_{\cC_k}$ for the mean value of $\phi$ on the set $\cC_k$.

\begin{lemma}
\label{Lem: Non-local pressure estimate}
Let $\mu$ satisfy Assumption~\ref{Ass: Coefficients} for some constants $\mu_{\bullet} , \mu^{\bullet} > 0$. Let $\lambda \in \IC$ and let for $f \in \L^2_{\sigma} (\IR^d)$ and $F \in \L^2 (\IR^d ; \IC^{d \times d})$ the functions $u \in \H^1_{\sigma} (\IR^d)$ and $\phi \in \L^2_{\loc} (\IR^d)$ solve
\begin{align*}
 \left\{ \begin{aligned}
  \lambda u - \divergence \mu \nabla u + \nabla \phi &= f + \divergence(F) && \text{in } \IR^d, \\
  \divergence (u) &= 0 && \text{in } \IR^d
 \end{aligned} \right.
\end{align*}
in the sense of distributions. Let $x_0 \in \IR^d$ and $r > 0$ let $\cC_0$ denote the ball $B(x_0 , r)$. Then there exists a constant $C > 0$ depending only on $\mu^{\bullet}$ and $d$ such that for all $k \in \IN$ we have
\begin{align*}
 &\bigg(\int_{\cC_k} \lvert \phi - \phi_{\cC_k} \rvert^2 \, \d x \bigg)^{\frac{1}{2}} \\
 &\qquad\leq C \bigg( \sum_{\ell = 0}^{k - 2} 2^{\frac{d}{2} (\ell - k)} \big( \| \nabla u \|_{\L^2 (\cC_{\ell})} + \| F \|_{\L^2 (\cC_{\ell})} \big) + \sum_{\substack{\ell \in \IN_0 \\ \lvert \ell - k \rvert \leq 1}} \big( \| \nabla u \|_{\L^2 (\cC_{\ell})} + \| F \|_{\L^2 (\cC_{\ell})} \big) \\
 &\qquad\qquad + \sum_{\ell = k + 2}^{\infty} 2^{(\frac{d}{2} + 1) (k - \ell)} \big( \| \nabla u \|_{\L^2 (\cC_{\ell})} + \| F \|_{\L^2 (\cC_{\ell})} \big) \bigg).
\end{align*}
\end{lemma}

The final preparatory result we need is a local Caccioppoli inequality that includes the pressure function.

\begin{lemma}
\label{Lem: Caccioppoli}
Let $\mu$ satisfy Assumption~\ref{Ass: Coefficients} for some constants $\mu_{\bullet} , \mu^{\bullet} > 0$. Then there exists $\omega \in (\pi / 2 , \pi)$ such that for all $\theta \in (0 , \omega)$ there exists $C > 0$ such that for all $x_0 \in \IR^d$, $r > 0$, $c \in \IC$, and all solutions $u \in \cH^1_{\sigma} (B(x_0 , 2 r))$ and $\phi \in \L^2 (B(x_0 , 2 r))$ (in the sense of distributions) to
\begin{align*}
 \left\{ \begin{aligned}
  \lambda u - \divergence \mu \nabla u + \nabla \phi &= 0 && \text{in } B(x_0 , 2 r), \\
  \divergence (u) &= 0 && \text{in } B(x_0 , 2 r)
 \end{aligned} \right.
\end{align*}
satisfy
\begin{align*}
 \lvert \lambda \rvert \int_{B(x_0 , r)} \lvert u \rvert^2 \, \d x + \int_{B(x_0 , r)} \lvert \nabla u \rvert^2 \, \d x \leq \frac{C}{r^2} \int_{B(x_0 , 2 r)} \lvert u \rvert^2 \, \d x + \frac{C}{\lvert \lambda \rvert r^2} \int_{B(x_0 , 2 r)} \lvert \phi - c \rvert^2 \, \d x.
 \end{align*}
The constant $C$ only depends on $d$, $\theta$, $\mu_{\bullet}$, and $\mu^{\bullet}$.
\end{lemma}

\begin{proof}
Let $\eta \in \C_c^{\infty} (B(x_0 , 2 r))$ with $\eta \equiv 1$ in $B(x_0 , r)$, $0 \leq \eta \leq 1$, and $\| \nabla \eta \|_{\L^{\infty}} \leq 2 / r$. Applying~\cite[Lem.~5.1]{Tolksdorf} with $c_1 = c$ and $c_2 = 0$ implies that
\begin{align*}
 &\lvert \lambda \rvert \int_{B(x_0 , 2 r)} \lvert u \eta \rvert^2 \, \d x + \int_{B(x_0 , 2 r)} \lvert \nabla [u \eta] \rvert^2 \, \d x \\
 &\qquad\leq \frac{C}{r^2} \int_{B(x_0 , 2 r)} \lvert u \rvert^2 \, \d x + \frac{4}{r} \bigg( \int_{B(x_0 , 2 r) \setminus B(x_0 , r)} \lvert \phi - c \rvert^2 \, \d x \bigg)^{\frac{1}{2}} \bigg( \int_{B(x_0 , 2 r)} \lvert u \eta \rvert^2 \, \d x \bigg)^{\frac{1}{2}}.
\end{align*}
Use Young's inequality to estimate
\begin{align*}
 \frac{4}{r} \bigg( \int_{B(x_0 , 2 r) \setminus B(x_0 , r)} \lvert \phi - c \rvert^2 \, \d x \bigg)^{\frac{1}{2}} \bigg( \int_{B(x_0 , 2 r)} \lvert u \eta \rvert^2 \, \d x \bigg)^{\frac{1}{2}} &\leq \frac{8}{\lvert \lambda \rvert r^2} \int_{B(x_0 , 2 r) \setminus B(x_0 , r)} \lvert \phi - c \rvert^2 \, \d x \\
 &\qquad + \frac{\lvert \lambda \rvert}{2} \int_{B(x_0 , 2 r)} \lvert u \eta \rvert^2 \, \d x.
\end{align*}
The lemma follows by absorbing the $u \eta$-term to the left-hand side and by using the properties of $\eta$. Finally, we would like to mention that the proof of~\cite[Lem.~5.1]{Tolksdorf} follows the standard proof that is used to establish the Caccioppoli inequality for elliptic systems and this is well-known.
\end{proof}

The following theorem presents $\L^2$ off-diagonal type estimates for the resolvent operators.

\begin{theorem}
\label{Thm: Off-diagonal for resolvent}
There exists $\omega \in (\pi / 2 , \pi)$ such that for all $\theta \in (0 , \omega)$ and all $\nu \in (0 , 2)$ there exists a constant $C > 0$ such that for all $x_0 \in \IR^d$, $r > 0$, $\lambda \in \S_{\theta}$, $f \in \L^2_{\sigma} (\IR^d)$, and $F \in \L^2 (\IR^d ; \IC^{d \times d})$ the unique solution $u \in \H^1_{\sigma} (\IR^d)$ to~\eqref{Eq: Resolvent problem} satisfies
\begin{align*}
 \int_{B(x_0 , r)} \lvert \lambda u \rvert^2 \, \d x + \int_{B(x_0 , r)} \lvert \lvert \lambda \rvert^{\frac{1}{2}} \nabla u \rvert^2 \, \d x &\leq C \int_{B(x_0 , 2 r)} \big( \lvert f \rvert^2 + \lvert \lvert \lambda \rvert^{\frac{1}{2}} F \rvert^2 \big) \, \d x \\
 &+ C \sum_{k = 2}^{\infty} \bigg( \frac{1}{1 + \lvert \lambda \rvert 2^{2 k} r^2} \bigg)^{\frac{\nu}{2}} \int_{B(x_0 , 2^k r)} \big( \lvert f \rvert^2 + \lvert \lvert \lambda \rvert^{\frac{1}{2}} F \rvert^2 \big) \, \d x.
\end{align*}
Here, the constant $C$ only depends on $d$, $\theta$, $\nu$, $\mu_{\bullet}$, and $\mu^{\bullet}$ and $\omega$ only depends on $d$, $\mu_{\bullet}$, and $\mu^{\bullet}$.
\end{theorem}

\begin{proof}
Fix $f \in \L^2_{\sigma} (\IR^d)$, $F \in \L^2 (\IR^d ; \IC^{d \times d})$, and $\lambda \in \S_{\theta}$. Define $u := (\lambda + A)^{-1} (f + \IP \divergence(F))$ and let $\phi \in \L^2_{\loc} (\IR^d)$ be the associated pressure such that $u$ and $\phi$ solve~\eqref{Eq: Resolvent problem}. Let $x_0 \in \IR^d$ and $r > 0$. In the following, we consider two cases. \par
Let $\lambda$ and $r$ be such that $\lvert \lambda \rvert r^2 \leq 1$. In this case, Theorem~\ref{Thm: Nonlocal resolvent estimate} yields the estimate
\begin{align*}
 \int_{B(x_0 , r)} \lvert \lambda u \rvert^2 \, \d x &\leq C \sum_{k = 0}^{\infty} 2^{- \nu k} \int_{B(x_0 , 2^k r)} \big( \lvert f \rvert^2 + \lvert \lvert \lambda \rvert^{\frac{1}{2}} F \rvert^2 \big) \, \d x \\
 &\leq 2^{\frac{\nu}{2}} C \sum_{k = 0}^{\infty} \bigg( \frac{1}{1 + \lvert \lambda \rvert 2^{2 k} r^2} \bigg)^{\frac{\nu}{2}} \int_{B(x_0 , 2^k r)} \big( \lvert f \rvert^2 + \lvert \lvert \lambda \rvert^{\frac{1}{2}} F \rvert^2 \big) \, \d x.
\end{align*}

Thus, it is left to consider the case $\lvert \lambda \rvert r^2 > 1$. In this case, define $g := f|_{B(x_0 , 2r)}$ and $G := F|_{B(x_0 , 2r)}$. The definition of $\Lop^2_{\sigma} (B(x_0 , 2r))$ implies that $g \in \Lop^2_{\sigma} (B(x_0 , 2r))$. Then, there exists $u_1 \in \cH^1_{\sigma} (B(x_0 , 2r))$ such that for all $v \in \cH^1_{\sigma} (B(x_0 , 2r))$ it holds
\begin{align*}
 \lambda \int_{B(x_0 , 2 r)} u_1 \cdot \overline{v} \; \d x + \fb_{B(x_0 , 2 r)} (u_1 , v) = \int_{B(x_0 , 2 r)} g \cdot \overline{v} \; \d x - \int_{B(x_0 , 2 r)} G_{\alpha \beta} \cdot \overline{\partial_{\alpha} v_{\beta}} \; \d x.
\end{align*}
Let $\phi_1 \in \L^2 (B(x_0 , 2 r))$ denote the associated pressure. By~\eqref{Eq: Resolvent estimate Neumann} we find that
\begin{align}
\label{Eq: Estimate u1}
\begin{aligned}
 \| \lambda u_1 \|_{\L^2 (B(x_0 , 2 r))} + \lvert \lambda \rvert^{\frac{1}{2}} \| \nabla u_1 \|_{\L^2 (B (x_0 , 2 r))} &+ \lvert \lambda \rvert^{\frac{1}{2}} \| \phi_1 \|_{\L^2 (B(x_0 , 2 r))} \\
 &\qquad\leq C \big( \| f \|_{\L^2 (B (x_0 , 2 r))} + \lvert \lambda \rvert^{\frac{1}{2}} \| F \|_{\L^2 (B(x_0 , 2 r))} \big).
 \end{aligned}
\end{align}
Notice that the constant $C > 0$ only depends on $d$, $\theta$, $\mu_{\bullet}$, and $\mu^{\bullet}$. In particular, it does not depend on $x_0$ and $r$. \par
Now, define $u_2 := u - u_1$ and $\phi_2 := \phi - \phi_1$. Thus, to prove the desired result, we only have to control $u_2$ in $B(x_0 , r)$. By definitions of all functions, we find that
\begin{align*}
 \lambda \int_{B (x_0 , 2 r)} u_2 \cdot \overline{v} \; \d x + \fb_{B (x_0 , 2 r)} (u_2 , v) - \int_{B (x_0 , 2 r)} \phi_2 \, \overline{\divergence(v)} \; \d x = 0 \qquad (v \in \H^1_0 (B (x_0 , 2 r) ; \IC^d)),
\end{align*}
so that by virtue of Lemma~\ref{Lem: Caccioppoli} we have
\begin{align*}
 &\int_{B(x_0 , r)} \lvert \lambda u_2 \rvert^2 \, \d x + \int_{B(x_0 , r)} \lvert \lvert \lambda \rvert^{\frac{1}{2}} \nabla u_2 \rvert^2 \, \d x \\
 &\qquad\qquad\leq \frac{C \lvert \lambda \rvert}{r^2} \int_{B(x_0 , 2 r)} \lvert u_2 \rvert^2 \, \d x + \frac{C}{r^2} \int_{B(x_0 , 2 r) \setminus B(x_0 , r)} \lvert \phi_2 - \phi_{B(x_0 , 2 r) \setminus B(x_0 , r)} \rvert^2 \, \d x.
\end{align*}
Now, use that $u_2 = u - u_1$ and $\phi_2 = \phi - \phi_1$ followed by~\eqref{Eq: Estimate u1}, Lemma~\ref{Lem: Non-local pressure estimate}, and $\nu < 2 < 2 + d$ to deduce that
\begin{align*}
 &\int_{B(x_0 , r)} \lvert \lambda u_2 \rvert^2 \, \d x + \int_{B(x_0 , r)} \lvert \lvert \lambda \rvert^{\frac{1}{2}} \nabla u_2 \rvert^2 \, \d x \\
 &\leq \frac{C \lvert \lambda \rvert}{r^2} \int_{B(x_0 , 2 r)} \lvert u_2 \rvert^2 \; \d x + \frac{C}{r^2} \int_{B(x_0 , 2r)} \lvert \phi_1 \rvert^2 \, \d x  + \frac{C}{r^2} \int_{B(x_0 , 2 r) \setminus B(x_0 , r)} \lvert \phi - \phi_{B(x_0 , 2 r) \setminus B(x_0 , r)} \rvert^2 \, \d x \\
 &\leq \frac{C}{\lvert \lambda \rvert r^2} \int_{B(x_0 , 2 r)} \lvert f \rvert^2 \, \d x + \frac{C \lvert \lambda \rvert}{r^2} \int_{B(x_0 , 2 r)} \lvert u \rvert^2 \, \d x + \frac{C}{r^2} \int_{B(x_0 , 2 r) \setminus B(x_0 , r)} \lvert \phi - \phi_{B(x_0 , 2 r) \setminus B(x_0 , r)} \rvert^2 \, \d x \allowdisplaybreaks \\
 &\leq \frac{C}{\lvert \lambda \rvert r^2} \int_{B(x_0 , 2 r)} \lvert f \rvert^2 \, \d x + C \bigg( \frac{\lvert \lambda \rvert}{r^2} \int_{B(x_0 , 2 r)} \lvert u \rvert^2 \, \d x + \frac{1}{r^2} \sum_{k = 0}^{\infty} 2^{- \nu k} \int_{B(x_0 , 2^{k + 1} r)} \lvert \nabla u \rvert^2 \, \d x\bigg) \\
 &\qquad + \frac{C}{\lvert \lambda \rvert r^2} \sum_{k = 0}^{\infty} 2^{- \nu k} \int_{B(x_0 , 2^{k + 1} r)} \lvert \lvert \lambda \rvert^{\frac{1}{2}} F \rvert^2 \, \d x.
\end{align*}
Now, employ Theorem~\ref{Thm: Non-local Caccioppoli} to the second term on the right-hand side followed by the non-local resolvent estimate in Theorem~\ref{Thm: Nonlocal resolvent estimate} so as to get
\begin{align*}
 \int_{B(x_0 , r)} \lvert \lambda u_2 \rvert^2 \, \d x &+ \int_{B(x_0 , r)} \lvert \lvert \lambda \rvert^{\frac{1}{2}} \nabla u_2 \rvert^2 \, \d x \\
 &\leq \frac{C}{\lvert \lambda \rvert r^2} \int_{B(x_0 , 2 r)} \lvert f \rvert^2 \, \d x + \frac{C}{\lvert \lambda \rvert^2 r^4} \sum_{k = 0}^{\infty} 2^{- \nu k} \int_{B(x_0 , 2^{k + 1} r)} \big( \lvert f \rvert^2 + \lvert \lvert \lambda \rvert^{\frac{1}{2}} F \rvert^2 \big) \, \d x \\
 &\qquad + \frac{C}{\lvert \lambda \rvert r^2} \sum_{k = 0}^{\infty} 2^{- \nu k} \int_{B(x_0 , 2^{k + 1} r)} \big( \lvert f \rvert^2 + \lvert \lvert \lambda \rvert^{\frac{1}{2}} F \rvert^2 \big) \, \d x.
\end{align*}
Finally, using that $\lvert \lambda \rvert r^2 > 1$ and $\nu < 2$, we get
\begin{align*}
 \int_{B(x_0 , r)} \lvert \lambda u_2 \rvert^2 \, \d x &+ \int_{B(x_0 , r)} \lvert \lvert \lambda \rvert^{\frac{1}{2}} \nabla u_2 \rvert^2 \, \d x \leq C\sum_{k = 2}^{\infty} \bigg( \frac{1}{\lvert \lambda \rvert 2^{2 k} r^2} \bigg)^{\frac{\nu}{2}} \int_{B(x_0 , 2^k r)} \big( \lvert f \rvert^2 + \lvert \lvert \lambda \rvert^{\frac{1}{2}} F \rvert^2 \big) \, \d x. \qedhere
\end{align*}
\end{proof}

\begin{remark}
\label{Rem: Gradient of resolvent}
We just proved slightly more than stated in Theorem~\ref{Thm: Off-diagonal for resolvent}. Indeed, if $\lvert \lambda \rvert r^2 > 1$, we proved further estimates on $\nabla u$ that are given by
\begin{align*}
 \lvert \lambda \rvert^{\frac{1}{2}} \| \nabla (\lambda + A)^{-1} f \|_{\L^2 (B(x_0 , 2 r))} \leq C \| f \|_{\L^2 (B(x_0 , 2 r))} + C\sum_{k = 2}^{\infty} \bigg( \frac{1}{\lvert \lambda \rvert 2^{2 k} r^2} \bigg)^{\frac{\nu}{4}} \| f \|_{\L^2(B(x_0 , 2^k r))}
\end{align*}
and
\begin{align*}
 \| \nabla (\lambda + A)^{-1} \IP \divergence(F) \|_{\L^2 (B(x_0 , 2 r))} \leq C \| F \|_{\L^2 (B(x_0 , 2 r))} + C\sum_{k = 2}^{\infty} \bigg( \frac{1}{\lvert \lambda \rvert 2^{2 k} r^2} \bigg)^{\frac{\nu}{4}} \| F \|_{\L^2 (B(x_0 , 2^k r))}.
\end{align*}
\end{remark}

\section{Estimates on the generalized Stokes semigroup}

\noindent Since $A$ satisfies the resolvent estimates
\begin{align*}
 \lvert \lambda \rvert \| (\lambda + A)^{-1} f \|_{\L^2} \leq C \| f \|_{\L^2} \qquad (\lambda \in \S_{\omega}),
\end{align*}
for some $\omega \in (\pi / 2 , \pi)$ the generalized Stokes operator $- A$ is the infinitesimal generator of a bounded analytic semigroup $(\e^{- t A})_{t \geq 0}$ which is represented via the Cauchy integral formula
\begin{align}
\label{Eq: Cauchy integral}
 \e^{- t A} = \frac{1}{2 \pi \ii} \int_{\gamma_t} \e^{t \lambda} (\lambda + A)^{-1} \, \d \lambda \qquad (t > 0).
\end{align}
Here, the path $\gamma_t$ runs through $\partial (B(0 , t^{-1}) \cup S_{\vartheta})$ for some $\vartheta \in (\pi / 2 , \omega)$ in a counterclockwise manner. This representation by the Cauchy integral formula allows to transfer estimates on the resolvent to estimates on the semigroup. For example, it is well-known that the estimates~\eqref{Eq: Resolvent estimate} and~\eqref{Eq: Resolvent estimate 2} used within~\eqref{Eq: Cauchy integral} directly yield for all $f \in \L^2_{\sigma} (\IR^d)$, $F \in \L^2 (\IR^d ; \IC^{d \times d})$, and $t > 0$ the semigroup estimates
\begin{align}
\label{Eq: Semigroup estimate 1}
 \| \e^{- t A} f \|_{\L^2} + t^{\frac{1}{2}} \|\nabla \e^{- t A} f \|_{\L^2} + t \| A \e^{- t A} f \|_{\L^2} \leq C \| f \|_{\L^2}
\end{align}
and
\begin{align}
\label{Eq: Semigroup estimate 2}
 t^{\frac{1}{2}} \| \e^{- t A} \IP \divergence(F) \|_{\L^2} + t \|\nabla \e^{- t A} \IP \divergence(F) \|_{\L^2} \leq C \| F \|_{\L^2}.
\end{align}
The following proof of Theorem~\ref{Thm: Generalized Stokes semigroup} shows that this transfer of estimates is also valid for the resolvent estimates established in Theorem~\ref{Thm: Off-diagonal for resolvent}.

\begin{proof}[Proof of Theorem~\ref{Thm: Generalized Stokes semigroup}]
Let $f \in \L^2_{\sigma} (\IR^d)$ and $F \in \L^2 (\IR^d ; \IC^{d \times d})$. Combining the conclusion of Theorem~\ref{Thm: Off-diagonal for resolvent} with~\eqref{Eq: Cauchy integral} directly yields for $x_0 \in \IR^d$ and $r > 0$ that
\begin{align*}
 &\| \e^{- t A} (f + \IP \divergence(F)) \|_{\L^2 (B(x_0 , r))} \\
 &\leq \frac{1}{2 \pi} \int_{\gamma_t} \e^{t \Re(\lambda)} \big\{ \| (\lambda + A)^{-1} (f + \IP \divergence (F)) \|_{\L^2 (B(x_0 , r))} \\
 &\qquad + t^{\frac{1}{2}} \| \nabla (\lambda + A)^{-1} (f + \IP \divergence(F)) \|_{\L^2 (B(x_0 , r))} \big\} \, \lvert \d \lambda \rvert \\
 &\leq C \int_{\gamma_t} \e^{t \Re(\lambda)} (\lvert \lambda \rvert^{-1} + t^{\frac{1}{2}} \lvert \lambda \rvert^{- \frac{1}{2}}) (\| f \|_{\L^2 (B(x_0 , 2^k r))} + \| \lvert \lambda \rvert^{\frac{1}{2}} F \|_{\L^2(B(x_0 , 2^k r))}) \, \lvert \d \lambda \rvert \\
 &+ C \sum_{k = 2}^{\infty} \int_{\gamma_t} \e^{t \Re(\lambda)} \bigg( \frac{1}{1 + \lvert \lambda \rvert 2^{2 k} r^2} \bigg)^{\frac{\nu}{4}} (\lvert \lambda \rvert^{-1} + t^{\frac{1}{2}} \lvert \lambda \rvert^{- \frac{1}{2}}) (\| f \|_{\L^2 (B(x_0 , 2^k r))} + \| \lvert \lambda \rvert^{\frac{1}{2}} F \|_{\L^2(B(x_0 , 2^k r))}) \, \lvert \d \lambda \rvert.
\end{align*}
Now, perform the substitution $\lambda t = \mu$ and use that for $\mu \in \gamma_1$ one has
\begin{align*}
 \frac{1}{1 + \frac{\lvert \mu \rvert 2^{2k} r^2}{t}} \leq \frac{1}{1 + \frac{2^{2k} r^2}{t}}\cdotp
\end{align*}
This readily yields that
\begin{align*}
 &\| \e^{- t A} (f + \IP \divergence(F)) \|_{\L^2 (B(x_0 , r))} \\
 &\leq C \sum_{k = 0}^{\infty} \bigg( \frac{1}{1 + \frac{2^{2k} r^2}{t}} \bigg)^{\frac{\nu}{4}} \int_{\gamma_1} \e^{\Re(\mu)} (\lvert \mu \rvert^{-1} + \lvert \mu \rvert^{- \frac{1}{2}}) (\| f \|_{\L^2 (B(x_0 , 2^k r))} + t^{- \frac{1}{2}} \| \lvert \mu \rvert^{\frac{1}{2}} F \|_{B(x_0 , 2^k r)}) \, \lvert \d \mu \rvert
\end{align*}
and thus already the desired estimate. \par
To estimate $t A \e^{- t A} (f + \IP \divergence(F))$, notice that
\begin{align*}
 A \e^{- t A} = \frac{1}{2 \pi \ii} \int_{\gamma_t} \e^{t \lambda} A (\lambda + A)^{-1} \, \d \lambda &= \frac{1}{2 \pi \ii} \int_{\gamma_t} \e^{t \lambda} (\Id - \lambda (\lambda + A)^{-1}) \, \d \lambda \\
 &= - \frac{1}{2 \pi \ii} \int_{\gamma_t} \lambda \e^{t \lambda} (\lambda + A)^{-1} \, \d \lambda.
\end{align*}
Now, the desired estimate follows analogously as above.
\end{proof}

\begin{remark}
If we assume that $r^2 / t > 1$, then all $\lambda \in \gamma_t$ satisfy $\lvert \lambda \rvert r^2 > 1$ so that in this case the estimates from Remark~\ref{Rem: Gradient of resolvent} together with the proof of Theorem~\ref{Thm: Generalized Stokes semigroup} yield the following gradient estimate on the generalized Stokes semigroup: there exists a constant $C > 0$ such that for all $f \in \L^2_{\sigma} (\IR^d)$, $F \in \L^2 (\IR^d ; \IC^{d \times d})$, and all $t > 0$ we have
\begin{align*}
 t^{\frac{1}{2}} \| \nabla \e^{- t A} f \|_{\L^2 (B(x_0 , r))} \leq C \| f \|_{\L^2 (B(x_0 , 2 r))} + C \sum_{k = 2}^{\infty} \bigg( \frac{2^{2 k} r^2}{t} \bigg)^{- \frac{\nu}{4}} \| f \|_{\L^2 (B(x_0 , 2^k r))}
\end{align*}
and
\begin{align*}
 \| \nabla \e^{- t A} F \|_{\L^2 (B(x_0 , r))} \leq C \| F \|_{\L^2 (B(x_0 , 2 r))} + C \sum_{k = 2}^{\infty} \bigg( \frac{2^{2 k} r^2}{t} \bigg)^{- \frac{\nu}{4}} \| F \|_{\L^2 (B(x_0 , 2^k r))}.
\end{align*}
\end{remark}

\begin{proof}[Proof of Corollary~\ref{Cor: Generalized Stokes semigroup}]
We distinguish two cases. Assume first that $2^{2 k_0} r^2 / t < 1$. Then by using the global $\L^2$-estimates~\eqref{Eq: Semigroup estimate 1}, we find that
\begin{align*}
 \| \e^{- t A} f \|_{\L^2 (B(x_0 , r))} + t \| A \e^{- t A} f \|_{\L^2 (B(x_0 , r))} &\leq C \| f \|_{\L^2(\IR^d)} \\
 &\leq 2^{\frac{\nu}{4}} \bigg( 1 + \frac{2^{2 k_0} r^2}{t} \bigg)^{- \frac{\nu}{4}} \| f \|_{\L^2 (B(x_0 , 2^{k_0} r) \setminus B(x_0 , 2^{k_0 - 1} r))}.
\end{align*}
Now, assume that $2^{2 k_0} r^2 / t \geq 1$. Then Theorem~\ref{Thm: Generalized Stokes semigroup} implies that
\begin{align*}
 \| \e^{- t A} f \|_{\L^2 (B(x_0 , r))} &+ t \| A \e^{- t A} f \|_{\L^2 (B(x_0 , r))} \\
 &\qquad\leq \sum_{k = k_0}^{\infty} \bigg( 1 + \frac{2^{2 k} r^2}{t} \bigg)^{- \frac{\nu}{4}} \| f \|_{\L^2 (B(x_0 , 2^{k_0} r) \setminus B(x_0 , 2^{k_0 - 1} r))} \\
 &\qquad\leq \bigg( \frac{2^{2 k_0} r^2}{t} \bigg)^{- \frac{\nu}{4}} \sum_{k = k_0}^{\infty} 2^{- \frac{\nu}{2} (k - k_0)} \| f \|_{\L^2 (B(x_0 , 2^{k_0} r) \setminus B(x_0 , 2^{k_0 - 1} r))} \\
 &\qquad\leq C \bigg( 1 + \frac{2^{2 k_0} r^2}{t} \bigg)^{- \frac{\nu}{4}} \| f \|_{\L^2 (B(x_0 , 2^{k_0} r) \setminus B(x_0 , 2^{k_0 - 1} r))}.
\end{align*}

To estimate the terms involving $\e^{- t A} \IP \divergence(F)$ proceed similarly, but by employing~\eqref{Eq: Semigroup estimate 2} in the first case and Theorem~\ref{Thm: Generalized Stokes semigroup} in the second case. We omit further details.
\end{proof}


\begin{bibdiv}
\begin{biblist}

\bibitem{Aronson}
D.~G.~Aronson.
\newblock Bounds for the fundamental solution of a parabolic equation.
\newblock Bull.\@ Amer.\@ Math.\@ Soc.~\textbf{73} (1967), 890--896.

\bibitem{Auscher}
P.~Auscher.
\newblock Regularity theorems and heat kernel for elliptic operators.
\newblock J.\@ London Math.\@ Soc.~(2) \textbf{54} (1996), no.~2, 284--296.

\bibitem{Auscher_Memoirs}
P.~Auscher.
\newblock On necessary and sufficient conditions for $L^p$-estimates of Riesz transforms associated to elliptic operators on $\IR^n$ and related estimates.
\newblock Mem.\@ Amer.\@ Math.\@ Soc.~~\textbf{186} (2007), no.~871.

\bibitem{Auscher_Frey}
P.~Auscher and D.~Frey.
\newblock On the well-posedness of parabolic equations of Navier-Stokes type with $\mathrm{BMO}^{-1}$ data.
\newblock J.\@ Inst.\@ Math.\@ Jussieu~\textbf{16} (2017), no.~5, 947--985.

\bibitem{Auscher_Hofmann_Lacey_McIntosh_Tchamitchian}
P.~Auscher, S.~Hofmann, M.~Lacey, A.~McIntosh, and P.~Tchamitchian.
\newblock The solution of the Kato square root problem for second order elliptic operators on $\IR^n$.
\newblock Ann.\@ of Math.~(2) \textbf{156} (2002), no.~2, 633--654.

\bibitem{Davies}
E.~B.~Davies.
\newblock {\em Heat kernels and spectral theory\/}.
\newblock Cambridge University Press, Cambridge, 1990.

\bibitem{Davies_trick}
E.~B.~Davies.
\newblock Uniformly elliptic operators with measurable coefficients.
\newblock J.\@ Funct.\@ Anal.~\textbf{132} (1995), no.~1, 141--169.

\bibitem{Davies_example}
E.~B.~Davies.
\newblock Limits on $L^p$ regularity of self-adjoint elliptic operators.
\newblock J.\@ Differential Equations~\textbf{135} (1997), no.~1, 83--102.

\bibitem{Frehse}
J.~Frehse.
\newblock An irregular complex valued solution to a scalar uniformly elliptic equation.
\newblock Calc.\@ Var.\@ Partial Differential Equations~\textbf{33} (2008), no.~3, 263--266.

\bibitem{Koch_Tataru}
H.~Koch and D.~Tataru.
\newblock Well-posedness for the Navier-Stokes equations.
\newblock Adv.\@ Math.~\textbf{157} (2001), no.~1, 22--35.

\bibitem{Mazya_Nazarov_Plamenevskii}
V.~G.~Maz'ya, S.~A.~Nazarov, and B.~A.~Plamenevskii.
\newblock Absence of De Giorgi-type theorems for strongly elliptic equations with complex coefficients.
\newblock J.\@ Math.\@ Sov.~\textbf{28} (1985), 726--739.

\bibitem{Tolksdorf}
P.~Tolksdorf.
\newblock A non-local approach to the generalized Stokes operator with bounded measurable coefficients.
\newblock Available on \url{arXiv:2011.13771}.

\end{biblist}
\end{bibdiv}
\end{document}